\documentclass{article}
\usepackage{amsmath,amsthm,amssymb,amsfonts,mathtools,enumerate,enumitem,mathrsfs}
\usepackage[round]{natbib} 

\newcommand\norm[1]{\left\lVert#1\right\rVert}

\theoremstyle{plain}
\newtheorem{theorem}{Theorem}

\theoremstyle{definition}

\theoremstyle{remark}
\newtheorem{remark}{Remark}

\newtheoremstyle{assumption-style}%
{0pt}%space above
{0pt}%space below
{}% body font
{}% paragraph indent amount
{\bfseries}% head font
{.}% punctuation after head
{.5em}% space after head
{}% headspec
\theoremstyle{assumption-style}
\newtheorem{assumption}{\textbf{A}\ignorespaces}
\newenvironment{Qtheorem}[1][]
  {\quote\begin{assumption}[#1]}
  {\end{assumption}\ \endquote}

\usepackage{hyperref}
\hypersetup{
    colorlinks,
    citecolor=black,
    filecolor=black,
    linkcolor=black,
    urlcolor=blue
}
\usepackage{indentfirst}
\numberwithin{equation}{section}
\allowdisplaybreaks

\newcommand{\R}{\mathbb R}

\def\assumptionautorefname{}

\def\assumptionautorefname{}

\makeatletter
\renewcommand{\assumptionautorefname}{A\@gobble}
\makeatother

\setlist{nolistsep}

\setenumerate[1]{label=\Roman*.}
\setenumerate[2]{label=\Alph*.}
\setenumerate[3]{label=\roman*.}
\setenumerate[4]{label=\alph*.}

\begin{document}

\title{On Convergence of a Truncation Scheme for Approximating Stationary Distributions of
Continuous State Space Markov Chains and Processes}
\author{Alex Infanger\thanks{Institute for Computational \& Mathematical Engineering, Stanford University, USA.} \and Peter W. Glynn\footnotemark[1]\ \thanks{Department of Management Science \& Engineering, Stanford University, USA.}}
\maketitle

\begin{abstract} \noindent 
  In the analysis of Markov chains and processes, it is sometimes convenient to replace an unbounded state space with a ``truncated'' bounded state space. When such a replacement is made, one often wants to know whether the equilibrium behavior of the truncated chain or process is close to that of the untruncated system. For example, such questions arise naturally when considering numerical methods for computing stationary distributions on unbounded state space. In this paper, we use the principle of ``regeneration'' to show that the stationary distributions of ``fixed state'' truncations converge in great generality (in total variation norm) to the stationary distribution of the untruncated limit, when the untruncated chain is positive Harris recurrent. Even in countable state space, our theory extends known results by showing that the augmentation can correspond to an $r$-regular measure. In addition, we extend our theory to cover an important subclass of Harris recurrent Markov processes that include non-explosive Markov jump processes on countable state space.
\end{abstract}

%!TEX root = ./FixedStateConvergencePaper.tex
\section{Introduction} 

When $X = (X_n: n \geq 0)$ is a positive recurrent Markov chain taking values in an infinite state space $S$, the question of how to approximate its stationary distribution $\pi$ becomes relevant. For a discrete state space Markov chain with one-step transition matrix $P = (P(x,y) : x,y \in S)$, a natural approximation scheme involves truncating the state space $S$ to a finite subset $A_n$, followed by construction of an approximating transition matrix $P_n = (P_n(x,y): x,y \in A_n)$. When $P_n$ is a stochastic matrix for which $P_n(x,y) \ge P(x,y)$ for $x,y\in A_n$, we say that  $P_n$ is an augmentation of the truncation $A_n$. However, if $\pi_n$ is a stationary distribution associated with $P_n$, it is well known that $\pi_n$ can fail to converge to $\pi$, even when $A_n$ increases to $S$ when $n$ tends to infinity; see, for example, (2.5) in \citet{wolfApproximationInvariantProbability1980}.

However, when one constructs the augmentation via a “fixed state” augmentation, \citet{wolfApproximationInvariantProbability1980} showed that convergence of $\pi_n$ to $\pi$ is guaranteed when $X$ is an irreducible positive recurrent countable state space Markov chain. This result also appears in \citet{gibsonAugmentedTruncationsInfinite1987}. In this paper, we show that such convergence can be generally validated when the underlying Markov chain or process is suitably regenerative; see \citet{thorissonCouplingStationarityRegeneration2000} for an extensive discussion of regeneration. The result essentially follows via a monotone convergence argument based on the representation of the stationary distribution (also known as the equilibrium distribution, steady-state distribution, or invariant measure) in terms of a ratio of regenerative ``cycle'' quantities. In order to implement this idea, the proof requires coupling the dynamics of the truncated chain to the dynamics of the untruncated chain $X$, so that the truncated and untruncated chains have identical trajectories over a regenerative cycle, up to the moment when the truncated chain attempts to exit the truncation set. This coupling argument lies at the heart of the proofs that form the core contribution of this paper.

Even in countable state space, our coupling-based proof approach differs from the existing theory, and clarifies the circumstances under which such ``fixed state augmentation'' (involving returning the truncated chain to a fixed state $z$ every time the chain attempts to exit the truncation set) can be generalized to a ``fixed return distribution augmentation'' (involving redistributing the truncated chain according to a fixed distribution $\nu$ every time the chain attempts to exit the truncation set). In particular, the theory requires that $\nu$ be $r$-regular; see Theorems \ref{thm::generalized-fixed-distribution-convergence} and \ref{thm::m-step-minorization-convergence} for details. However, the main advantage of this proof approach is that it allows us to easily develop convergence results for general state space Markov chains and also general state space Markov jump processes. With these theorems available, computational algorithms for stationary distributions can focus on effective computational methods for Markov chains and processes taking values in a compact state space, provided that the compact state space truncation is of the form described in this paper.

\autoref{sec::strongly-minorized-Harris-recurrent-fixed-state-convergence} describes the coupling underlying our proof technique, in the special setting of Harris recurrent Markov chains satisfying a “strong minorization” condition; see \autoref{assumption::Lyapunov-with-strong-minorization}. In that setting, we show that the approximating stationary distribution $\pi_n$ converges to $\pi$ in a very strong sense, namely in the weighted total variation distance, provided that a suitable Lyapunov condition is in force. The Lyapunov condition is essentially necessary and sufficient, given the total variation distance we are using. We also study conditions under which one can extend the augmentation to fixed return distribution augmentations; see \autoref{thm::generalized-fixed-distribution-convergence}. In \autoref{sec::relaxing-strong-minorization}, we describe a more complicated coupling suitable for general Harris recurrent Markov chains (that do not necessarily exhibit strong minorizations). \autoref{thm::m-step-minorization-convergence} provides a general convergence result in the weighted total variation norm for fixed return distribution augmentations for such chains. Finally, \autoref{sec::fixed-state-theory-for-Markov-jump-processes} describes the corresponding theory for a certain class of Harris recurrent Markov jump processes that possess regenerative structure permitting direct application of our discrete time arguments to proving convergence. According to our understanding (see, for example, Section 4.3.4 of of \citet{kuntzStationaryDistributionsContinuoustime2021}), this continuous time theory provides the first rigorous proof of convergence, even for fixed state truncation/augmentation, when dealing with non-explosive Markov jump processes having unbounded jump rates.

In \citet{infanger2022convergence}, arguments utilizing tightness and weak continuity of the transition probabilities are used to study related, completely general (non fixed state) truncation-augmentation algorithms for countable state and continuous state Markov chains. In contrast, the current paper exploits regeneration rather than tightness, so that the conditions imposed here are quite different. In addition, while the results in the current paper allow the truncation to be very general, the theory based on tightness imposes conditions on the $A_n$'s, namely that the $A_n$'s be sublevel sets of a suitably chosen Lyapunov function.

\section{Convergence Theory for Strongly Minorized Harris Recurrent Markov Chains}\label{sec::strongly-minorized-Harris-recurrent-fixed-state-convergence} 
Let $X=(X_n:n\geq 0)$ be an $S$-valued Markov chain, with one-step transition kernel $P=(P(x,dy):x,y\in S)$. Here $P(x,dy)=P(X_{n+1}\in dy|X_n=x)$ for $x,y\in S$. For a generic non-negative function $w$, let $Pw$ be the real-valued function for which $(Pw)(x)=\int_{S}^{}w(y)P(x,dy)$. Also, for a (measurable) subset $B\subseteq S$, let $I_B:S\rightarrow\{0,1\}$ be the function for which $I_B(x)$ is 1 or 0 depending whether or not $x\in B$. Finally, for $x\in S$, let $\delta_x(\cdot)$ be the probability associated with a unit point mass at $x\in S$.

Our main assumption in this section is: 
\begin{Qtheorem}\label{assumption::Lyapunov-with-strong-minorization}
There exists a non-empty subset $C\subseteq S$, $g:S\rightarrow \R_+$,$r:S\rightarrow[1,\infty)$, $\lambda>0$, $b\in \R_+$, and a probability $\phi$ on $C$ such that:
\begin{enumerate}[label=\roman*)]
\item $\sup_{x\in C}g(x)<\infty$; \label{assumption::Lyapunov-with-strong-minorization-bounded-on-C}
\item $(Pg)(x)\leq g(x)-r(x)+bI_C(x),x\in S$;
\item $P(x,dy)\geq \lambda\phi(dy),x,y\in C$.\label{assumption::Lyapunov-with-strong-minorization-strong-minorization}
\end{enumerate}
\end{Qtheorem}

We refer to condition \autoref{assumption::Lyapunov-with-strong-minorization} \ref{assumption::Lyapunov-with-strong-minorization-strong-minorization} as the \emph{strong minorization} condition, because the inequality involves the one-step transition kernel (rather than the $m$-step transition kernel). In \autoref{sec::relaxing-strong-minorization}, we weaken this hypothesis. Condition \autoref{assumption::Lyapunov-with-strong-minorization} \ref{assumption::Lyapunov-with-strong-minorization-strong-minorization} ensures that $C$ is a \emph{small set}; see p.102 of \citet{meynMarkovChainsStochastic2009}. We refer to $g$ as a stochastic \emph{Lyapunov function} for $X$.

\begin{remark} When $S$ is countably infinite, $X$ is irreducible, and there exists $z$ for which $P(z,z)>0$, then \autoref{assumption::Lyapunov-with-strong-minorization} holds for any Markov chain with stationary distribution $\pi=(\pi(x):x\in S)$ for which 
\begin{align*}
\sum_{x\in S}^{}\pi(x)r(x)<\infty.
\end{align*}
We may take $C=\{z\}$, $\lambda=P(z,z)$, $\phi(\cdot)=\delta_z(\cdot)$, and 
\begin{align*}
g(x)=E_x\sum_{j=0}^{T(C)-1} r(X_j),
\end{align*}
where $T(B)=\inf\{k\geq 0: X_k\in B\}$ is the first hitting time of $B$ for a generic subset $B\subseteq S$; see Chapter 14 of \citet{meynMarkovChainsStochastic2009} (see, in particular, the last point in the proof of part (ii) of Theorem 14.2.3).
\end{remark}
\begin{remark} We can write
\begin{align*}
P(x,dy) = p(x,y)\phi(dy)+ P_S(x,dy)
\end{align*}
for $x\in C,y\in S$, where $P_S(x,\cdot)$ is singular with respect to $\phi$. The condition \autoref{assumption::Lyapunov-with-strong-minorization} asserts that the transition density $p(\cdot)$ can be chosen to satisfy
\begin{align*}
\inf_{x,y\in C} p(x,y)\geq \lambda.
\end{align*}
Thus, \autoref{assumption::Lyapunov-with-strong-minorization} \ref{assumption::Lyapunov-with-strong-minorization-strong-minorization} holds when $X$ has a transition density that is bounded below over $C$. 
\end{remark}

For $x\in S$, let $P_x(\cdot)=P(\cdot|X_0=x)$ and $E_x(\cdot)$ be its associated expectation. Put $S(B)=\inf\{k\geq 1: X_k\in B\}$ for a generic subset $B$. \autoref{assumption::Lyapunov-with-strong-minorization} implies that
\begin{subequations}\label{eq::sum-of-rewards-until-TC-bound}
\begin{align}
E_x \sum_{j=0}^{T(C)-1}r(X_j)\leq g(x)\label{eq::sum-of-rewards-until-TC-bound-a}
\end{align}
for $x\in C^c$ and
\begin{align}
\sup_{x\in C}E_x\sum_{j=0}^{S(C)-1}r(X_j)\leq b+\sup_{x\in C}g(x);\label{eq::sum-of-rewards-until-TC-bound-b}
\end{align}
see p.344 of \citet{meynMarkovChainsStochastic2009}. Since $r(x)\geq 1$, \eqref{eq::sum-of-rewards-until-TC-bound}  implies that $X$ hits the small set $C$ infinitely often $P_x$ a.s., and hence is a Harris recurrent Markov chain. In fact, \eqref{eq::sum-of-rewards-until-TC-bound-b} implies that $E_xS(C)$ is bounded in $x$ over $C$ and hence $X$ is a positive recurrent Harris chain; see p.266 of \citet{meynMarkovChainsStochastic2009}. It follows that $X$ has a unique stationary distribution $\pi$ for which
\begin{align*}
\pi(dy) = \int_{S}^{}\pi(dx) P(x,dy)
\end{align*}
for $y\in S$. 
\end{subequations}

\citet{athreyaNewApproachLimit1978} and \citet{nummelinSplittingTechniqueHarris1978} observed that \autoref{assumption::Lyapunov-with-strong-minorization} \ref{assumption::Lyapunov-with-strong-minorization-strong-minorization} can be used to write the transition probabilities on $C$ as a mixture of two distributions, namely
\begin{align*}
P(x,\cdot) = \lambda\phi(\cdot) + (1-\lambda)H(x,\cdot)
\end{align*}
for $x\in C$, where $H(x,\cdot)=(P(x,\cdot)-\lambda\phi(\cdot))/(1-\lambda)$. (When $\lambda=1$, $H(x,\cdot)$ may be chosen arbitrarily.) Informally, when $X$ hits $C$ (say at time $n$), then $X$ distributes itself according to $\phi$ at time $n+1$ with probability $\lambda$ and distributes itself according to $H(X_n, \cdot)$ with probability $1-\lambda$. At times at which $X$ distributes itself according to $\phi$, $X$ regenerates. We note that since $X$ visits $C$ infinitely often and regenerates with probability $\lambda$ at each such visit, we may conclude that there are infinitely many such regeneration times at which $X$ has distribution $\phi$; see p.324 of \citet{doobStochasticProcesses1953}.

A more precise description requires enriching our probability space so that it supports a sequence $(\beta_n:n\geq 0)$ of random variables, in which $\beta_n=1$ when $X$ regenerates at time $n$ and is 0 otherwise; see \citet{athreyaNewApproachLimit1978} and \citet{nummelinGeneralIrreducibleMarkov1984} for details. If $\tau=\inf\{k\geq 1: \beta_k=1\}$, then $\tau$ is the first regeneration time, and $\tau_k=\inf\{j>\tau_{k-1}: \beta_j=1\}$ is the time of the $k$'th regeneration.

We let $P_\mu(\cdot)=\int_{S}^{}\mu(dx)P_x(\cdot)$ be the probability under which $X$ is initialized with distribution $\mu$, and let $E_\mu(\cdot)$ be its associated expectation. If $S_1(B)=S(B)$ and $S_k(B)=\inf\{j>S_{k-1}(B):X_j\in B\}$ for $k\geq 2$, note that if
\begin{align*}
s &= \sup_{x\in C} E_x \sum_{j=0}^{S(C)-1}r(X_j) <\infty,\\
\gamma &=\sup_{x\in C} E_x\sum_{j=0}^{\tau-1}r(X_j),
\end{align*}
then, for $x\in C$,
\begin{align*}
E_x\sum_{j=0}^{\tau-1} r(X_j)&\leq E_x \sum_{j=0}^{S_2(C)-1}r(X_j) + E_x I(\tau>S(C)+1)\sum_{j=S_2(C)}^{\tau-1} r(X_j)\\
&\leq 2s + E_x I(\tau>S(C)+1)E_x[\sum_{j=S_2(C)}^{\tau-1} r(X_j)|X_i, \beta_i, i\leq S_2(C)]\\
&\leq 2s+E_x I(\tau>S(C)+1)\gamma\\
&=2s + \gamma E_x(1-P_x(\tau=S(C)+1|X_{S(C)}))\\
&= 2s+\gamma(1-\lambda),
\end{align*}
and hence
\begin{align*}
\gamma\leq 2s+(1-\lambda)\gamma.
\end{align*}
Consequently, 
\begin{align}
E_\phi \sum_{j=0}^{\tau-1} r(X_j)\leq \gamma\leq 2s/\lambda,\label{eq::finite-mean-r-cycle-length}
\end{align}
so $X$ is a regenerative process with finite mean cycle length. As a result, the law of large numbers for regenerative processes (see, for example, \citet{smithRegenerativeStochasticProcesses1955}) implies that for each non-negative $f$,

\begin{align}
\frac{1}{n}\sum_{j=0}^{n-1}E_\phi f(X_j) \rightarrow \frac{E_\phi \sum_{j=0}^{\tau-1} f(X_j)}{E_\phi \tau} = \int_{S}^{}m(dx) f(x) \overset{\Delta}{=} mf\label{eq::regenerative-process-measure}
\end{align}
where $m(\cdot)$ is the probability given by 
\begin{align*}
m(\cdot) = \frac{E_\phi \sum_{j=0}^{\tau-1}I(X_j\in\cdot)}{E_\phi \tau}.
\end{align*}
Applying \eqref{eq::regenerative-process-measure} to $Pf$ we find that
\begin{align}
\frac{1}{n}\sum_{j=0}^{n-1}E_\phi (Pf)(X_j)\rightarrow mPf\label{eq::mPf-convergence}
\end{align}
for $f$ non-negative. But
\begin{align*}
E_\phi (Pf)(X_j) = E_\phi f(X_{j+1}),
\end{align*}
so 
\begin{align}
\frac{1}{n}\sum_{j=0}^{n-1}E_\phi (Pf)(X_j)= \frac{1}{n}\sum_{j=1}^{n}E_\phi f(X_j).\label{eq::Pfandf-sum}
\end{align}
If $f$ is also bounded, it follows from \eqref{eq::regenerative-process-measure}, \eqref{eq::mPf-convergence}, and \eqref{eq::Pfandf-sum} that $mf=mPf$ for all bounded and non-negative $f$, thereby implying that $m=mP$. Hence, $m$ is the stationary distribution of $P$, i.e.
\begin{align}
\pi(\cdot) = \frac{E_\phi \sum_{j=0}^{\tau-1}I(X_j\in\cdot)}{E_\phi \tau}.\label{eq::phi-regenerative-representation}
\end{align}

We now assume a sequence $(A_n:n\geq 1)$ of truncation sets for which $C\subseteq A_1 \subseteq A_2\subseteq ...$ with $\bigcap_{n=1}^\infty A_n^c=\emptyset$ with $P_\phi(T_n<S(C))>0$ for $n\geq 1$. Note that if $T_n=\inf\{k\geq 0: X_k\in A_n^c\}$ is the exit time from $A_n$, then $P_x(T_n\rightarrow\infty \text{ as }n\rightarrow\infty)=1$ for each $x\in S$. We first consider the augmentation in which $X$ is forced to re-enter $A_n$ using re-entry distribution $\phi$ whenever it attempts to exit $A_n$. In other words, the transition kernel $P_n=(P_n(x,dy):x,y\in A_n)$ governing $X$ under the $n$'th truncation-augmentation is given by
\begin{align}
P_n(x,dy) = P(x,dy) + P(x,A_n^c)\phi(dy)\label{eq::regenerative-augmentation}
\end{align}
for $x,y\in A_n$ for $n\geq 1$. Of course, we can re-write \eqref{eq::regenerative-augmentation} as a mixture, namely 
\begin{align}
P_n(x,dy) = P(x,A_n)R_n(x,dy) +P(x,A_n^c)\phi(dy)\label{eq::mixture-version-of-regenerative-augmentation}
\end{align}
for $x,y\in A_n$. As in our earlier discussion in this section, we can use the $\beta_n$'s to randomize the choice between $R_n$ and $\phi$. In particular, when $X_k\in A_n-C$, we let $\beta_{k+1}=2$ with probability $P(X_k, A_n^c)$, in which case $X_{k+1}$ has distribution $\phi$. Otherwise, $\beta_{k+1}=0$ and $X_{k+1}$ has distribution $R_n(X_k,\cdot)$. On the other hand, if $X_k\in C$, $\beta_{k+1}$ again equals 2 with probability $P(X_k, A_n^c)$, in which case $X_{k+1}$ has distribution $\phi$. With probability $\lambda$, $\beta_{k+1}=1$ and $X_{k+1}$ has distribution $\phi$, while $X_{k+1}$ has distribution $(P(X_k, \cdot \cap A_n)-\lambda\phi(\cdot))/(1-\lambda-P(X_k, A_n^c))$ and $\beta_{k+1}=0$ otherwise (with probability $1-\lambda-P(X_k, A_n^c)$). Let $\Gamma=\inf\{k\geq 1: \beta_k=2\}$ be the first time at which $X$ attempts to exit $A_n$ under the transition kernel $P_n$, and let $\tilde P_{x,n}(\cdot)$ be the associated probability on the path-space of the $(X_i, \beta_i)$'s with $\tilde E_{x,n}(\cdot)$ its corresponding expectation. Futher, note that if $\tau=\inf\{k\geq 1:\beta_k\geq 1\}$, $X_\tau$ has distribution $\phi$ and $\tau$ is a regeneration time for $X$ under both $\tilde P_{\phi, n}$ and $P_{\phi}$.

If $x_j\in A_n$ and $i_j\in\{0,1\}$ for $0\leq j\leq k$, then 
\begin{align}
&\ \ \ \ \ \ P_{x_0}(X_1\in dx_1, \beta_1=i_1, X_2\in dx_2, \beta_2=i_2, ..., X_k\in dx_k, \beta_k=i_k, T_n>k)\nonumber\\ 
&=\ \  P_{x_0}(X_1\in dx_1, \beta_1=i_1, X_2\in dx_2,\beta_2=i_2,...,X_k\in dx_k, \beta_k=i_k)\nonumber\\
&=\ \  \tilde P_{x_0, n}(X_1\in dx_1, \beta_1=i_1, X_2\in dx_2, \beta_2=i_2,...,X_k\in dx_k, \beta_k=i_k)\nonumber\\
&=\ \  \tilde P_{x_0, n}(X_1\in dx_1, \beta_1=i_1, X_2\in dx_2; \beta_2=i_2,..., X_k\in dx_k, \beta_k=i_k, \Gamma>k)\label{eq::equality-in-distribution-until-exit}.
\end{align}

For distributions $\mu_1$ and $\mu_2$ defined on $S$, let
\begin{align*}
\norm{\mu_1-\mu_2}_{r} = \sup\left\{\int_{S}^{}(\mu_1-\mu_2)(dx)f(x):|f(x)|\leq r(x), x\in S\right\}
\end{align*}
be the $r$-weighted total variation norm distance between $\mu_1$ and $\mu_2$.

We are now ready to state our convergence theorem for the truncation-\\augmentation scheme defined by \eqref{eq::mixture-version-of-regenerative-augmentation}. 

\begin{theorem}\label{thm-1} Assume \autoref{assumption::Lyapunov-with-strong-minorization} and suppose that $P_n$ is the truncated-augmented transition kernel defined by \eqref{eq::mixture-version-of-regenerative-augmentation}. Then for each $n\geq 1$, $X$ is a positive recurrent Harris chain with a unique stationary distribution $\pi_n$, and 
\begin{align*}
\norm{\pi_n-\pi}_{r}\rightarrow 0
\end{align*}
as $n\rightarrow\infty$.
\end{theorem}

\begin{proof} We first note that $C\subseteq A_1$, and that $X_\Gamma$ has distribution $\phi$ under $\tilde P_{x,n}$. Furthermore, $\phi(C)=1$, so $X_\Gamma\in C$. Consequently, $T(C)=T(C)\land \Gamma$ under $\tilde P_{x,n}$, so that for $x\in A_n$, \eqref{eq::equality-in-distribution-until-exit} implies that
\begin{align*}
\tilde E_{x,n}T(C)&= \sum_{k=0}^{\infty}\tilde P_{x,n}(X_0\not\in C, ..., X_k\not\in C)\\
&= \sum_{k=0}^{\infty} \tilde P_{x,n} (X_0\not\in C,...,X_k\not\in C, \Gamma>k)\\
&= \sum_{k=0}^{\infty} P_x(X_0\not\in C, ...,X_k\not\in C, T_n>k)\\
&=\sum_{k=0}^{\infty}P_x(T_n\land T(C)>k)\\
&= E_x T_n\land T(C)\\
&\leq E_xT(C) \leq g(x).
\end{align*}
Consequently, \autoref{assumption::Lyapunov-with-strong-minorization} holds with $P_n$ replacing $P$ and $r(x)\equiv 1$ for $x\in S$. Hence, $X$ is positive Harris recurrent with a unique stationary distribution $\pi_n$ under the transition kernel $P_n$. 

As with the transition kernel $P$, the Markov chain $X$ regenerates under $P_n$ whenever $X$ is distributed according to $\phi$ (i.e. at times $\tau_i$ at which $\beta_{\tau_i}\in\{1,2\}$). The ratio representation formula for regenerative processes therefore applies, so that as with \eqref{eq::phi-regenerative-representation},
\begin{align*}
\pi_n(\cdot) = \frac{\tilde E_{\phi, n}\sum_{j=0}^{\tau-1} I(X_j\in \cdot)}{\tilde E_{\phi, n}\tau},
\end{align*}
where $\tilde P_{\mu, n}(\cdot)=\int_{S}^{}\mu(dx)P_{x,n}(\cdot)$ and $\tilde E_{\mu, n}(\cdot)$ is its associated expectation (for a generic distribution $\mu$). The equality \eqref{eq::equality-in-distribution-until-exit} implies that
\begin{align*}
P_{x}(X_k\in \cdot, \tau\land T_n>k) &= \tilde P_{x,n}(X_k\in \cdot, \tau\land \Gamma>k)\\
&=\tilde P_{x,n}(X_k\in \cdot, \tau>k)
\end{align*}
and hence 
\begin{align*}
E_x \sum_{k=0}^{(\tau\land T_n)-1}I(X_k\in \cdot) = \tilde E_{x,n} \sum_{k=0}^{\tau-1}I(X_k\in \cdot),
\end{align*}
so that we can represent $\pi_n$ as 
\begin{align}
\pi_n(\cdot) &=\frac{E_\phi \sum_{j=0}^{(\tau\land T_n)-1}I(X_j\in \cdot)}{E_{\phi} \tau\land T_n}.\label{eq::regenerative-representation-for-pi-n-in-terms-of-Ephi}
\end{align}
It follows from \eqref{eq::phi-regenerative-representation} and \eqref{eq::regenerative-representation-for-pi-n-in-terms-of-Ephi} that for $f$ satisfying $|f(x)|\leq r(x)$ for $x\in S$, we can write
\begin{align*}
\int_{S}^{}(\pi_n-\pi)(dx)f(x)= \frac{E_\phi \sum_{j=0}^{(\tau\land T_n)-1} f(X_j)}{E_{\phi}\tau\land T_n} - \frac{E_\phi \sum_{j=0}^{\tau-1}f(X_j)}{E_\phi \tau}.
\end{align*}
Since $r(x)\geq 1$ for $x\in S$, it follows from \eqref{eq::finite-mean-r-cycle-length} that
\begin{align*}
\biggr|\int_{S}^{}&(\pi_n-\pi)(dx)f(x)\biggr|  \\
&=\left|\frac{E_\phi (\tau-(\tau\land T_n))E_\phi \sum_{j=0}^{(\tau\land T_n)-1}f(X_j)- E_\phi \sum_{j=\tau\land T_n}^{\tau-1}f(X_j)\cdot E_{\phi} \tau\land T_n}{\left(E_\phi \tau\land T_n\right)\left(E_\phi \tau\right)}\right|\\
&\leq \frac{2 E_{\phi} \sum_{j=\tau\land T_n}^{\tau-1}r(X_j) E_\phi \sum_{j=0}^{\tau-1} r(X_j)}{(E_{\phi}\tau\land T_n)^2}\rightarrow 0
\end{align*}
as $n\rightarrow\infty$, uniformly over $f$ satisfying $|f(x)|\leq r(x),x\in S,$ proving the theorem.
\end{proof}

We now to turn to generalizing the class of augmentations that are permitted. In particular, we now modify our transition kernel $P_n$ so that
\begin{align}
P_n(x,dy) = P(x,dy) + P(x,A_n^c)\nu(dy)\label{eq::general-fixed-augmentation-distribution}
\end{align}
for $x,y\in A_n$, where $\nu(\cdot)$ is chosen so that $\nu(A_1)=1$. In other words, we now permit more general re-entry distributions. Such an augmentation is especially convenient, because unlike \eqref{eq::mixture-version-of-regenerative-augmentation}, it does not require finding a measure $\phi$ for which \autoref{assumption::Lyapunov-with-strong-minorization} holds.

\begin{Qtheorem}\label{assumption::r-regularity-of-nu} The distribution $\nu$ is chosen so that 
\begin{align}
E_\nu \sum_{j=0}^{T(C)-1}r(X_j)<\infty.\label{eq::r-regularity-of-nu}
\end{align}
We say that a distribution $\nu$ satisfying \eqref{eq::r-regularity-of-nu} is \emph{r-regular}; see p.271 and Chapter 14 \citet{meynMarkovChainsStochastic2009}.
\end{Qtheorem}

\begin{remark} Given \eqref{eq::sum-of-rewards-until-TC-bound}, it is evident that $\delta_x$ is $r$-regular for every $x\in A_1$. (In other words, each $x$ is an \emph{r-regular point} for $X$.)
\end{remark}

\begin{remark} In view of \eqref{eq::sum-of-rewards-until-TC-bound}, a sufficient condition for $\nu$ to be $r$-regular is that 
\begin{align*}
\int_{A_1}^{}g(x)\nu(dx)<\infty,
\end{align*}
where $g$ is the stochastic Lyapunov function of \autoref{assumption::Lyapunov-with-strong-minorization}.
\end{remark}

We can now state our convergence result for this more general setting. 

\begin{theorem}\label{thm::generalized-fixed-distribution-convergence} Assume \autoref{assumption::Lyapunov-with-strong-minorization} and \autoref{assumption::r-regularity-of-nu} and suppose that $P_n$ is the transition kernel given by \eqref{eq::general-fixed-augmentation-distribution}. Then, for $n$ sufficiently large, $X$ has a unique stationary distribution $\pi_n$ under $P_n$ satisfying
\begin{align*}
\norm{\pi_n-\pi}_{r}\rightarrow 0
\end{align*}
as $n\rightarrow\infty$. 
\end{theorem}
\begin{proof} In this setting, $\tau= \inf\{k\geq 1: \beta_k=1\}$ (rather than as in \autoref{thm-1}, where $\tau=\inf\{k\geq 1: \beta_k\in\{1,2\}\}$). In view of \eqref{eq::equality-in-distribution-until-exit}, 
\begin{align*}
\tilde E_{\nu, n} \sum_{j=0}^{\tau-1}I(X_j\in \cdot) &= \tilde E_{\nu, n} \sum_{j=0}^{(\tau\land \Gamma)-1}I(X_j\in \cdot) + \tilde E_{\nu, n} \sum_{j=\Gamma}^{\tau-1}I(X_j\in \cdot)\\
&=E_{\nu} \sum_{j=0}^{(\tau\land T_n)-1}I(X_j\in\cdot) + \tilde P_{\nu, n}(\Gamma<\tau)\tilde E_{\nu,n}\sum_{j=0}^{\tau-1}I(X_j\in\cdot)\\
&=E_\nu\sum_{j=0}^{(\tau\land T_n)-1}I(X_j\in\cdot) + P_\nu(T_n<\tau)\tilde E_{\nu, n}\sum_{j=0}^{\tau-1}I(X_j\in\cdot).
\end{align*}
Of course, $P_\nu(T_n<\tau)\searrow 0$ as $n\rightarrow\infty$, so $P_\nu(T_n<\tau)<1$ for $n$ sufficiently large. Hence, for $n$ sufficiently large,
\begin{align*}
\tilde E_{\nu,n}\sum_{j=0}^{\tau-1}I(X_j\in\cdot) = \frac{E_\nu \sum_{j=0}^{(T\land T_n)-1}I(X_j\in \cdot)}{P_{\nu}(\tau<T_n)}.
\end{align*}
Consequently, 
\begin{align}
\tilde E_{\phi, n}\sum_{j=0}^{\tau-1}I(X_j\in\cdot) &= \tilde E_{\phi,n} \sum_{j=0}^{(\tau\land \Gamma)-1}I(X_j\in\cdot) + \tilde E_{\phi,n} \sum_{j=\Gamma}^{\tau-1} I(X_j\in\cdot)\nonumber\\
&=E_\phi \sum_{j=0}^{(\tau\land T_n)-1}I(X_j\in\cdot) + P_\phi(T_n<\tau)\tilde E_{\nu,n}\sum_{j=0}^{\tau-1}I(X_j\in\cdot) \nonumber\\
&=\tilde E_{\phi}\sum_{j=0}^{(\tau\land T_n)-1}I(X_j\in\cdot) + P_\phi(T_n<\tau) \frac{E_\nu \sum_{j=0}^{(\tau\land T_n)-1}I(X_j\in\cdot)}{P_\nu(\tau<T_n)}.\label{eq::Ephi-in-terms-of-Enu}
\end{align}
Observe that
\begin{align*}
E_\nu\sum_{j=0}^{\tau-1}I(X_j\in\cdot) \leq E_\nu T(C)+ \sup_{x\in C}E_x\tau<\infty
\end{align*}
due to \autoref{assumption::r-regularity-of-nu} and \eqref{eq::finite-mean-r-cycle-length}. Applying the Monotone Convergence Theorem to \eqref{eq::Ephi-in-terms-of-Enu} then allows us to conclude that
\begin{align}
\tilde E_{\phi,n}\sum_{j=0}^{\tau-1}I(X_j\in B) \rightarrow E_\phi \sum_{j=0}^{\tau-1}I(X_j\in B)<\infty\label{eq::convergence-of-sums-of-indicators}
\end{align}
as $n\rightarrow\infty$ for each (measurable) $B$. As with \eqref{eq::phi-regenerative-representation}, this implies that
\begin{align*}
\pi_n(\cdot) =\frac{\tilde E_{\phi,n}\sum_{j=0}^{\tau-1}I(X_j\in\cdot)}{E_{\phi,n}\tau}
\end{align*}
is a stationary distribution for $X$ under $P_n$. Furthermore, \eqref{eq::convergence-of-sums-of-indicators} establishes that $\pi_n(B)\rightarrow \pi(B)$ as $n\rightarrow\infty$. In fact, the $r$-regularity implies that for $n$ large enough
\begin{align*}
\tilde E_{\phi,n}\sum_{j=0}^{\tau-1}r(X_j)\rightarrow E_\phi \sum_{j=0}^{\tau-1}r(X_j)<\infty
\end{align*}
as $n\rightarrow\infty$. The same argument as in the proof of \autoref{thm-1} then proves that
\begin{align*}
\norm{\pi_n-\pi}_{r}\rightarrow 0
\end{align*}
as $n\rightarrow\infty$.

Finally, we note that for $x\in A_n$ and $B$ such that $\phi(B)>0$, the Harris recurrence of $X$ under $P$ implies that either there is a set of paths from $x$ to $B$ having positive probability lying entirely within $A_n$, or $X$ attempts to exit to $A_n^c$ and is then re-distributed according to $\nu$ under $P_n$. But our above argument shows that $P_\nu(T_n<\tau)<1$ for $n$ sufficiently large, so that there is then also a set of paths of positive probability to $B$ from $x$ when $X$ attempts to exit to $A_n^c$. Hence, for $n$ sufficiently large, $X$ is a $\phi$-irreducible Markov chain, and consequently has a unique stationary distribution.
\end{proof}

\section{Relaxing the Strong Minorization Condition}\label{sec::relaxing-strong-minorization}
In this section, we require that $S$ be a separable metric space, and that the transition probabilities $(P(x,B):x\in S, B\in \mathscr{S})$ be defined relative to the Borel $\sigma$-algebra $\mathscr{S}$ associated with $S$. We make this assumption in order to ensure the existence of the conditional probability distributions that we will need later in this section; see p.79 in \citet{breimanProbabilityStochasticProcesses1969} and p.219 of \citet{billingsleyConvergenceProbabilityMeasures1968}. Given the state space $S$, we again let $(A_n:n\geq 1)$ be an increasing sequence of truncation sets, defined as in \autoref{sec::strongly-minorized-Harris-recurrent-fixed-state-convergence}.

We generalize the truncation-augmentation convergence theorems of \autoref{sec::strongly-minorized-Harris-recurrent-fixed-state-convergence} by weakening \autoref{assumption::Lyapunov-with-strong-minorization} as follows.

\begin{Qtheorem}\label{assumption::Lyapunov-with-m-step-minorization} There exists a non-empty subset $C\subseteq A_1,g:S\rightarrow \R_+,r :S\rightarrow[1,\infty), \lambda>0, b\in\R_+,m\geq 1,$ and a probability $\phi$ on $C$ such that:
\begin{enumerate}[label=\roman*)]
\item $\sup_{x\in C}g(x)<\infty$\label{assumption::Lyapunov-with-m-step-minorization::finite-g-in-C};
\item $(Pg)(x)\leq g(x)-r(x)+bI_C(x), x\in S$\label{assumption::Lyapunov-with-m-step-minorization::Lyapunov-function};
\item $P_x(X_m\in dy, T_1>m)\geq \lambda\phi(dy), x,y\in C$\label{assumption::Lyapunov-with-m-step-minorization::m-step-minorization}.
\end{enumerate}
\end{Qtheorem}

As compared to \autoref{assumption::Lyapunov-with-strong-minorization}, it is condition \ref{assumption::Lyapunov-with-m-step-minorization::m-step-minorization} that is weakened here to a minorization condition involving the $m$-step transition probabilities associated with having stayed within $A$ for all $m$ steps. 

\begin{remark}
We note that if the sub-stochastic kernel $(P(x,dy):x,y\in A_1)$ is $\phi$-irreducible, then existence of such a $C$-set is guaranteed; see p.7 of \citet{oreyLectureNotesLimit1971}. 
\end{remark}

For a given truncation set $A_n$, we define $P_n$ as in \eqref{eq::general-fixed-augmentation-distribution}, so that the distribution $\nu$ is used as the re-entry distribution whenever $X$ attempts to exit $A_n$. In the current setting, we require a more complex randomization than in \autoref{sec::strongly-minorized-Harris-recurrent-fixed-state-convergence}, because condition \autoref{assumption::Lyapunov-with-m-step-minorization} \ref{assumption::Lyapunov-with-m-step-minorization::m-step-minorization} involves the $m$-step transition kernel, taken over trajectories in which $T_1>m$. In particular, we note that \autoref{assumption::Lyapunov-with-m-step-minorization} \ref{assumption::Lyapunov-with-m-step-minorization::m-step-minorization} allows us to write, for $x\in C$, $y\in A_1$, 
\begin{align}
P_x(X_m\in dy, T_1>m) = \lambda \phi(dy) + (P_x(T_1>m)-\lambda)\tilde H(x,dy),\label{eq::law-of-Xm-intersect-with-T1-larger-than-m}
\end{align}
where $\tilde H(x,dy)=(P_x(X_m\in dy, T_1>m)-\lambda\phi(dy))/(P_x(T_1>m)-\lambda)$. Consequently, the right-hand side of \eqref{eq::law-of-Xm-intersect-with-T1-larger-than-m} is now expressed as a mixture of the probabilities $\phi(\cdot)$ and $\tilde H(x,\cdot)$, with corresponding mixture probabilities $\lambda$ and $P_x(T_1>m)-\lambda$. The first time $X$ visits $C$ at time $i$ (say), we generate $\beta_{i+m}$. When $\beta_{i+m}=1$ (with probability $\lambda$), we generate $X_{i+m}$ from distribution $\phi$. With probability $P_x(T_1>m)-\lambda,\beta_{i+m}=0$ and $X_{i+m}$ is generated from the distribution $\tilde H(x,\cdot)$. With $X_i$ and $X_{i+m}$ having been generated, we then ``fill in'' $(X_{i+1},...,X_{i+m-1})$ according to the appropriate conditional distribution, and set $\beta_j=0$ for $i+1\leq j\leq i+m$. The above description describes the randomization on those trajectories for which $X$ does not exit $A_1$ between $i$ and $i+m$. 

The full randomization procedure is described next. Given the existence of regular conditional distributions, we may define
\begin{align*}
M(dx_1, dx_2,...,dx_{m-1}|x,y) = P_x(X_i\in dx_i, 1\leq i\leq m-1|X_m=y, T_1>m)
\end{align*}
$P_x$ a.s., so that we may write (in view of \eqref{eq::equality-in-distribution-until-exit}) 
\begin{align}
\tilde P_{x,n} &(X_i\in dx_i, 1\leq i\leq m) = \nonumber\\
&\qquad \ \ \ \ \lambda \phi (dx_m) M(dx_1,...,dx_{m-1}|x,x_m)\nonumber\\
&\qquad +(P_x(T_1>m)-\lambda)\tilde H(x,dx_m) M(dx_1,...,dx_{m-1}|x,x_m)\nonumber\\
&\qquad +P_x(T_1\leq m<T_n)P_x(X_i\in dx_i, 1\leq i\leq m|T_1\leq m<T_n)\nonumber\\
&\qquad +P_x(T_n\leq m) \tilde P_{x,n}(X_i\in dx_i, 1\leq i\leq m|\Gamma\leq m)\label{eq::law-of-Xi-from-1-to-m}
\end{align}
for $x\in C$, where $\Gamma=\inf\{k\geq 1: \beta_k=2\}$ is the first time $X$ attempts to exit $A_n$ and is forced to re-enter $A_1$, using distribution $\nu$ (precisely as in \autoref{sec::strongly-minorized-Harris-recurrent-fixed-state-convergence}).

In other words, \eqref{eq::law-of-Xi-from-1-to-m} asserts that the general $m$-step transition structure under $P_n$, starting from $x\in C$, can be written as a mixture of four distributions, with corresponding mixture probabilities $\lambda, P_x(T_1>m)-\lambda, P_x(T_1\leq m < T_n)$, and $P_x(T_n\leq m)$. The $\beta$ random variable (rv) is set equal to 1 at time $m$ when mixture component 1 is used, and equal to 2 at attempted exit times from $A_n$ which can occur only within the path $(X_0,X_1,...,X_m)$ associated with mixture component 4. At all other times and within the other mixture components, $\beta$ equals 0.

Having generated $(X_1,...,X_m)$ using \eqref{eq::law-of-Xi-from-1-to-m}, we run the chain under the transition kernel $P_n$ until the next time $X$ re-visits $C$ at the time $r\geq m$ (say), thereby generating $(X_i,\beta_i)$ with $m<i\leq r$. (Again, $\beta_i=0$ except at those times $j$ at which $X$ attempts to exit $A_n$, in which case $\beta_j=2$.) We then again use the mixture representation \eqref{eq::law-of-Xi-from-1-to-m} to generate $((X_i,\beta_i): r<i\leq r+m)$, and continue repeating this process indefinitely to obtain $((X_k,\beta_k):k\geq 0)$. As in \autoref{sec::strongly-minorized-Harris-recurrent-fixed-state-convergence}, the times $\tau_1,\tau_2,...$ at which $\beta_{\tau_k}=1$ are regeneration times for $X$ under the transition kernel $P$ and under the transition kernels $P_n$ for $n\geq 1$. 

Note that for $x\in C$, 
\begin{align*}
E_x \sum_{j=0}^{\tau-1}r(X_j) &\leq E_x \sum_{j=0}^{S_{m+1}(C)-1}r(X_j) \\
&\qquad\qquad +E_x I(\tau>S_1(C)+m) \sum_{j=S_{m+1}(C)}^{\tau-1}r(X_j)\\
&\leq (m+1)s \ +\\
&\qquad E_xI(\tau>S_1(C)+m)E_x\bigr[\sum_{j=S_{m+1}(C)}^{\tau-1}r(X_j)|X_i, \beta_i:i\leq S_{m+1}(C)\bigr]\\
&\leq (m+1)s + \gamma P_x(\tau>S_1(C)+m)\\
&= (m+1)s + \gamma \lambda.
\end{align*}
Taking the supremum over $x\in C$ on the left-hand side, we conclude that
\begin{align*}
\gamma\leq (m+1)s + \gamma \lambda
\end{align*}
and hence 
\begin{align*}
\gamma\leq \frac{(m+1)s}{1-\lambda}<\infty,
\end{align*}
(since $s<\infty$ is implied by \autoref{assumption::Lyapunov-with-m-step-minorization} \ref{assumption::Lyapunov-with-m-step-minorization::finite-g-in-C} and \ref{assumption::Lyapunov-with-m-step-minorization::Lyapunov-function}; see \citet{meynMarkovChainsStochastic2009}).

Since $r(x)\geq 1$ for $x\in S$, it follows that $X$ is a positive recurrent regenerative process under the transition kernel $P$. The Markov chain $X$ is also guaranteed to be a positive recurrent Harris chain with stationary distribution $\pi$ given by the regenerative ratio formula \eqref{eq::phi-regenerative-representation}.

In the presence of \autoref{assumption::r-regularity-of-nu}, the same argument as used in \autoref{sec::strongly-minorized-Harris-recurrent-fixed-state-convergence} proves that $X$ is a positive recurrent regenerative process under the transition kernel $P_n$, for $n$ sufficiently large. Furthermore, the same argument as utilized in \autoref{sec::strongly-minorized-Harris-recurrent-fixed-state-convergence} proves that $X$ has a stationary distribution $\pi_n$ given by
\begin{align}
 \pi_n(\cdot) = \frac{E_\phi \sum_{j=0}^{(\tau\land T_n)-1}I(X_j\in \cdot) + P_\phi(T_n<\tau)E_\nu \sum_{j=0}^{(\tau\land T_n)-1}I(X_j\in \cdot)/P_\nu(\tau<T_n)}{E_\phi \tau\land T_n + P_\phi(T_n<\tau)(E_\nu \tau\land T_n)/P_\nu(\tau<T_n)}\label{eq::regenerative-representation-for-pi-n-for-m-step-minorization}
\end{align}
Finally, the argument used in \autoref{thm::generalized-fixed-distribution-convergence} proves the main result of this section.

\begin{theorem}\label{thm::m-step-minorization-convergence} Assume \autoref{assumption::r-regularity-of-nu} and \autoref{assumption::Lyapunov-with-m-step-minorization}, and suppose that $P_n$ is the transition kernel given by \eqref{eq::general-fixed-augmentation-distribution}. Then, for $n$ sufficiently large, $X$ has a unique stationary distribution $\pi_n$ under $P_n$ satisfying 
\begin{align*}
\norm{\pi_n-\pi}_{r}\rightarrow 0
\end{align*}
as $n\rightarrow\infty$.
\end{theorem}

\section{Extension to Markov Jump Processes}\label{sec::fixed-state-theory-for-Markov-jump-processes}
In this section, we extend the theory of our earlier sections to the setting of continuous time Markov jump processes. In particular, we assume that $X=(X(t):t\geq 0)$ is an $S$-valued non-explosive Markov jump process with right continuous piecewise constant sample paths; see \citet{meynStabilityMarkovianProcesses1993} for a Lyapunov condition that guarantees non-explosiveness. The dynamics of $X$ are determined by its rate transition kernel $Q=(Q(x,dy):x,y\in S)$, where we assume that $Q(x,dy)\geq 0$ for $x\neq y$,
\begin{align*}
0<\beta(x)\overset{\Delta}{=}\int_{S-\{x\}}^{}Q(x,dy)<\infty,
\end{align*}
and $Q(x,S)=0$ for $x\in S$. Then, $X$ has sample trajectories in which the sequence $Y=(Y_n: n\geq 0)$ of states visited by $X$ is a Markov chain in discrete time with transition kernel $R=(R(x,dy):x,y\in S)$ given by
\begin{align*}
R(x,dy) = \begin{cases} Q(x,dy)/\beta(x), & y\neq x\\
0, & y=x.
\end{cases}
\end{align*}

Also, the amount of time that $X$ spends in the $n$'th state visited is exponentially distributed with mean $1/\beta(Y_n)$.

The theory in continuous time is very similar to that which we have already developed in discrete time. Suppose that we are given a sequence $(A_n:n\geq 1)$ of truncation sets for which $A_1\subseteq A_2\subseteq...$ with $\bigcup_{n=1}^\infty A_n=S$. Assume that:

\begin{Qtheorem}\label{assumption::Markov-jump-process-Lyapunov-assumption} There exists a subset $\emptyset\neq C\subset A_1$, $\lambda>0$, a probability $\phi$ on $C$, and functions $g:S\rightarrow\mathbb R_+$, $r:S\rightarrow[1,\infty)$ for which
\begin{enumerate}[label=\roman*)]
\item $\sup_{x\in C}[\beta(x)g(x)+\beta^{-1}(x)+g(x)+\beta(x)(Rg)(x)]<\infty$\label{assumption::Markov-jump-process-Lyapunov-assumption::finiteness-of-quantities};
\item $\int_{S}^{}Q(x,dy)g(y)\leq -r(x)$ for $x\in C^c$;
\item $R(x,dy)\geq \lambda\phi(dy)$, $x\in C$, $y\in S$.
\end{enumerate}
\end{Qtheorem}
Such a jump process is guaranteed to have a unique stationary distribution $\pi$ (see CD2 of \citet{meynStabilityMarkovianProcesses1993}). Given a truncation set $A_n$, its associated rate transition kernel $Q_n=(Q_n(x,dy):x,y\in A_n)$ is characterized by
\begin{align*}
(Q_nf)(x) = \int_{A_n}^{}Q(x,dy)f(y) + Q(x,A_n^c)\int_{A_1}^{}\nu(dy)f(y)
\end{align*}
for $x\in A_n$ and $f$ bounded and measurable. Here, $\nu$ is a probability on $A_1$ for which 
\begin{align}
\int_{A_1}^{}g(x)\nu(dx)<\infty.\label{eq::fixed-state-mjp-g-is-nu-integrable}
\end{align}
We are now ready to state our continuous time analog to \autoref{thm::generalized-fixed-distribution-convergence}.

\begin{theorem}\label{thm::Markov-jump-process-fixed-state-convergence} Assume \autoref{assumption::Markov-jump-process-Lyapunov-assumption} and \eqref{eq::fixed-state-mjp-g-is-nu-integrable}. Then, $Q_n$ has a unique stationary distribution $\pi_n$ on $A_n$ for $n$ sufficiently large, and
\begin{align*}
\norm{\pi_n-\pi}_{r} \rightarrow 0
\end{align*}
as $n\rightarrow\infty$.
\end{theorem}
\begin{proof} The argument is similar to that used in the proof of \autoref{thm::generalized-fixed-distribution-convergence}. We first note that for each $x\in S$,
\begin{align*}
g(X(t))-\int_{0}^{t}(Qg)(X(s))ds
\end{align*}
is a $P_x$-local martingale, with localizing sequence $(T_n:n\geq 0)$ where $T_n=\inf\{t\geq 0:\beta(X(t))+(Rg)(X(t))+g(X(t))\geq n\}$. (Note that the conditional expectation of $g(X(T_n))$ is given by $(Rg)(X(T_n-))$, which is less than $n$.) Hence, \autoref{assumption::Markov-jump-process-Lyapunov-assumption} guarantees that
\begin{align}
g(X(t)) + \int_{0}^{t}r(X(s))ds - c\int_{0}^{t}I(X(s)\in C)ds\label{eq::Markov-jump-process-Px-local-martingale}
\end{align}
is a $P_x$-local supermartingale with respect to the same localizing sequence, where
\begin{align*}
c = \sup_{x\in C}[\beta(x)g(x)+\beta(x)(Rg)(x)].
\end{align*}
Let $\xi(C)=\inf\{t\geq 0: X(t)\in C\}$ be the first hitting time of $C$. By applying optional sampling to the supermartingale \eqref{eq::Markov-jump-process-Px-local-martingale} and utilizing the facts that $g(X(T_n\land \xi(C)\land t)$ is non-negative and that $X(s)\not\in C$ for $s<\xi(C)$, we conclude that
\begin{align*}
E_x \int_{0}^{t\land T_n\land \xi(C)}r(X(s))ds &\leq g(x)+cE_x\int_{0}^{t\land T_n\land \xi(C)}I(X(s)\in C)ds\\
&= g(x)
\end{align*}
for $x\in S$. Sending $t\rightarrow\infty$ and $n\rightarrow\infty$ and applying the Monotone Convergence Theorem, we find that
\begin{align*}
E_x\int_{0}^{\xi(C)}r(X(s))ds \leq g(x).
\end{align*}
Since $r(x)\geq 1$ for $x\in S$, it follows that $\xi(C)<\infty$ a.s. Consequently, $X$ visits $C$ infinitely often $P_x$ a.s. Note that the transition kernel of $Y$ under $Q_n$ can be written in the form
\begin{align*}
R_n(x,dy) = \lambda \phi(dy) + q_n(x)Q_n(x,dy) + r_n(x)\nu(dy)
\end{align*}
for $x\in C$, where $\lambda$, $q_n(x)$, and $r_n(x)$ are non-negative and sum to 1. As in our discrete-time arguments, $X$ regenerates every time $X$ visits $C$ and makes a state transition (with probability $\lambda$) according to $\phi$. Since $X$ visits $C$ infinitely often a.s. and there is a probability $\lambda$ of a regeneration on each such visit, the conditional Borel-Cantelli lemma (p.324 of \citet{doobStochasticProcesses1953}) establishes that regeneration will occur $P_x$ a.s. for each $x\in S$.

Applying optional sampling to the local supermartingale \eqref{eq::Markov-jump-process-Px-local-martingale}, we find that
\begin{align*}
E_x\int_{0}^{t\land T_n\land \tau}r(X(s))ds \leq g(x)+cE_x\int_{0}^{t\land T_n\land\tau}I(X(s)\in C)ds
\end{align*}
for $x\in S$, $t\geq 0$, and $n\geq 1$. Sending $t\rightarrow\infty$ and $n\rightarrow\infty$ and applying the Monotone Convergence Theorem, we conclude that
\begin{align}
E_x\int_{0}^{\tau}r(X(s))ds \leq g(x)+cE_x\int_{0}^{\tau}I(X(s)\in C)ds\label{eq::Markov-jump-process-bound-on-sum-of-rewards-until-regeneration}
\end{align}
for $x\in S$.

We note that on every visit to $C$, $X$ spends an exponentially distributed amount of time in $C$ with a mean no longer than $d\overset{\Delta}{=}\sup\{\beta(x)^{-1}:x\in C\}<\infty$ (by \autoref{assumption::Markov-jump-process-Lyapunov-assumption} \ref{assumption::Markov-jump-process-Lyapunov-assumption::finiteness-of-quantities}). Hence, the rate at which $X$ transitions out of any state in $C$ is at least $d^{-1}>0$. Consequently, the rate at which $X$ transitions to a regeneration is at least $\lambda d^{-1}$ while $X$ is in $C$. So,
\begin{align*}
P_x(X(s)\in C, \tau>s) \leq E_xI(X(s)\in C)\exp(-\lambda d^{-1} \int_{0}^{s}I(X(u)\in C)du)
\end{align*}
and
\begin{align}
E_x \int_{0}^{\tau}I(X(s)\in C)ds &= \int_{0}^{\infty}P_x(X(s)\in C, \tau>s)ds\nonumber\\
&\leq E_x\int_{0}^{\infty}I(X(s)\in C)\exp(-\lambda d^{-1}\int_{0}^{s}I(X(u)\in C)du)ds\nonumber\\
&\leq\frac{d}{\lambda}.\label{eq::bound-on-indicator-on-C-until-regeneration}
\end{align}
It follows from \eqref{eq::Markov-jump-process-bound-on-sum-of-rewards-until-regeneration} and \eqref{eq::bound-on-indicator-on-C-until-regeneration} that
\begin{align}
E_x\int_{0}^{\tau}r(X(s))ds \leq g(x)+\frac{cd}{\lambda}\label{eq::Markov-jump-process-bound-on-sum-of-rewards-until-regeneration-bound}
\end{align}
for $x\in S$. So \autoref{assumption::Markov-jump-process-Lyapunov-assumption} \ref{assumption::Markov-jump-process-Lyapunov-assumption::finiteness-of-quantities} implies that
\begin{align*}
E_\phi \int_{0}^{\tau}r(X(s))ds \leq \sup_{x\in C} g(x) + \frac{cd}{\lambda}<\infty,
\end{align*}
and hence the regenerative cycles of $X$ under $Q$ have finite expectation. So, the unique stationary distribution $\pi$ of $X$ can be expressed as
\begin{align*}
\pi(\cdot) = \frac{E_\phi \int_{0}^{\tau}I(X(s)\in \cdot)ds}{E_\phi \tau}.
\end{align*}
Furthermore, \eqref{eq::Markov-jump-process-bound-on-sum-of-rewards-until-regeneration-bound} establishes that
\begin{align*}
E_\nu\int_{0}^{\tau}r(X(s))ds \leq \int_{A_1}^{}\nu(dx)g(x)+\frac{cd}{\lambda}<\infty.
\end{align*}
The rest of the proof follows the same argument as that used in \autoref{thm::generalized-fixed-distribution-convergence}.
\end{proof}

\bibliographystyle{apalike}
\bibliography{FixedStateConvergence}

\clearpage

\end{document}